\documentclass[psamsfonts]{amsart} 
\usepackage{amssymb}

\usepackage[dvips]{graphicx}

\usepackage{graphicx}
\usepackage{amssymb}                       
\usepackage{amsmath}
\usepackage{color}
\usepackage{times}

\usepackage[unicode,bookmarks,colorlinks]{hyperref}
\hypersetup{
    linkcolor=brickred,
}

\definecolor{mahogany}{cmyk}{0, 0.77, 0.87, 0}
\definecolor{salmon}{cmyk}{0, 0.53, 0.38, 0}
\definecolor{melon}{cmyk}{0, 0.46, 0.50, 0}
\definecolor{yellowgreen}{cmyk}{0.44, 0, 0.74, 0}
\definecolor{brickred}{cmyk}{0, 0.89, 0.94, 0.28}
\definecolor{OliveGreen}{cmyk}{0.64, 0, 0.95, 0.40}
\definecolor{RawSienna}{cmyk}{0, 0.72, 1.0, 0.45}
\definecolor{ZurichRed}{rgb}{1, 0, 0} 

\usepackage{fancyhdr}
\usepackage{amsmath,amstext,amssymb,amsopn,amsthm}
\usepackage{amsmath,amssymb,amsthm}
\usepackage[mathscr]{eucal}

\begin{document}

\newtheorem{lemma}[thm]{Lemma}
\newtheorem{corr}[thm]{Corollary}
\newtheorem{proposition}{Proposition}
\newtheorem{theorem}{Theorem}[section]
\newtheorem{deff}[thm]{Definition}
\newtheorem{case}[thm]{Case}
\newtheorem{prop}[thm]{Proposition}
\newtheorem{example}{Example}

\newtheorem{corollary}{Corollary}

\theoremstyle{definition}
\newtheorem{remark}{Remark}

\numberwithin{equation}{section}
\numberwithin{definition}{section}
\numberwithin{corollary}{section}

\numberwithin{theorem}{section}

\numberwithin{remark}{section}
\numberwithin{example}{section}
\numberwithin{proposition}{section}

\newcommand{\calD}{\bD}

\newcommand{\conjugate}[1]{\overline{#1}}
\newcommand{\abs}[1]{\left| #1 \right|}
\newcommand{\cl}[1]{\overline{#1}}
\newcommand{\expr}[1]{\left( #1 \right)}
\newcommand{\set}[1]{\left\{ #1 \right\}}

\newcommand{\calC}{\mathcal{C}}
\newcommand{\calE}{\mathcal{E}}
\newcommand{\calF}{\mathcal{F}}
\newcommand{\Rd}{\mathbb{R}^d}
\newcommand{\BR}{\bD(\Rd)}
\newcommand{\R}{\mathbb{R}}
\newcommand{\al}{\alpha}
\newcommand{\RR}[1]{\mathbb{#1}}
\newcommand{\bR}{\mathrm{I\! R\!}}
\newcommand{\ga}{\gamma}
\newcommand{\om}{\omega}
\newcommand{\A}{\mathbb{A}}
\newcommand{\bH}{\mathbb{H}}

\newcommand{\bb}[1]{\mathbb{#1}}
\newcommand{\bI}{\bb{I}}
\newcommand{\bN}{\bb{N}}

\newcommand{\uS}{\mathbb{S}}
\newcommand{\M}{{\mathcal{M}}}
\newcommand{\calB}{{\mathcal{B}}}

\newcommand{\W}{{\mathcal{W}}}

\newcommand{\m}{{\mathcal{m}}}

\newcommand {\mac}[1] { \mathbb{#1} }

\newcommand{\bD}{\mathbb{D}}

\newcommand{\bC}{\Bbb C}

\newtheorem{rem}[theorem]{Remark}
\newtheorem{dfn}[theorem]{Definition}
\theoremstyle{definition}
\newtheorem{ex}[theorem]{Example}
\numberwithin{equation}{section}

\newcommand{\Pro}{\mathbb{P}}
\newcommand\F{\mathcal{F}}
\newcommand\E{\mathbb{E}}
\newcommand\e{\varepsilon}
\def\H{\mathcal{H}}
\def\t{\tau}

\newcommand{\blankbox}[2]{%
  \parbox{\columnwidth}{\centering
    \setlength{\fboxsep}{0pt}%
    \fbox{\raisebox{0pt}[#2]{\hspace{#1}}}%
  }%
}

\title[Fourier multipliers and Astala's theorem]{On Astala's theorem for martingales and Fourier multipliers}

\author{Rodrigo Ba\~nuelos}\thanks{R. Ba\~nuelos is supported in part  by NSF Grant
\# 0603701-DMS}
\address{Department of Mathematics, Purdue University, West Lafayette, IN 47907, USA}
\email{banuelos@math.purdue.edu}
\author{Adam Os\c ekowski}\thanks{A. Os\c ekowski is supported in part by Polish Ministry of Science
and Higher Education (MNiSW) grant IP2011 039571 `Iuventus Plus'}
\address{Department of Mathematics, Informatics and Mechanics, University of Warsaw, Banacha 2, 02-097 Warsaw, Poland}
\email{ados@mimuw.edu.pl}

\begin{abstract}
We exhibit a large class of symbols $m$ on $\R^d$, $d\geq 2$,  for which the corresponding Fourier multipliers $T_m$ satisfy the following inequality. If $D$, $E$ are measurable subsets of $\R^d$ with $E\subseteq D$ and $|D|<\infty$, then
$$ \int_{D\setminus E} |T_{m}\chi_E(x)|\mbox{d}x\leq \begin{cases}
 |E|+|E|\ln\left(\frac{|D|}{2|E|}\right), & \mbox{if }|E|<|D|/2,\\
 |D\setminus E|+\frac{1}{2}|D \setminus E|\ln \left(\frac{|E|}{|D\setminus E|}\right), & \mbox{if }|E|\geq |D|/2.  
\end{cases}$$
Here $|\cdot|$ denotes the Lebesgue measure on $\bR^d$. When $d=2$, these multipliers  include the real and imaginary parts of the Beurling-Ahlfors operator $B$ and hence the inequality is also valid for $B$ with the right-hand side multiplied by $\sqrt{2}$. The inequality is sharp for the real and imaginary parts of $B$.  This work is  motivated by K. Astala's celebrated results on the Gehring--Reich conjecture concerning  the distortion of area by quasiconformal maps.  The proof rests on probabilistic methods and exploits a family of appropriate novel sharp inequalities for differentially subordinate martingales.   
These martingale bounds are of interest on their own right.  
\end{abstract}

\maketitle

\section{Introduction}
The motivation for the results obtained in this paper come from an important estimate for the Beurling-Ahlfors transform $B$ on $\mathbb{C}$ proved by K. Astala in \cite{Ast}. 
This operator is a Fourier multiplier with the symbol $m(\xi)=\overline{\xi}/\xi$, $\xi\in \mathbb{C}$, and can also be defined as the  singular integral operator 
$$ Bf(z)=-\frac{1}{\pi}\,\mbox{p.v.}\int_\mathbb{C} \frac{f(w)}{(z-w)^2}\mbox{d}w.$$
A fundamental property of the Beurling-Ahlfors transform is that it changes the complex derivative $\overline{\partial}$ to $\partial$. More precisely, we have
\begin{equation}\label{intert}
 B(\overline{\partial} f)={\partial}f,
\end{equation}
for any $f$ in  the Sobolev space $ W^{1,2}(\mathbb{C},\mathbb{C})$ of complex valued locally integrable functions on $\mathbb{C}$, whose distributional first derivatives are in $L^2$ on the plane. 
The Beurling-Ahlfors operator plays a fundamental  role in the theory of quasiconformal
mappings in the plane, as detailed in \cite{AstIwaMar}. Recall that a homeomorphism $F : \mathbb{C}\to\mathbb{C}$ is said to
be $K$-quasiconformal, $K\geq 1$, if $F\in W_{loc}^{1,2}(\mathbb{C},\mathbb{C})$ and if $|\bar{\partial} F(z)|\leq k|\partial F(z)|$ for almost all $z\in \mathbb{C}$, where
$k = (K -1)/(K + 1)$. In the 50's, Bojarski \cite{Bo}, \cite{Bo2} applied the $L^p$-boundedness
of $B$ to prove that partial derivatives of $K$-quasiconformal maps, which are a priori locally square integrable, belong in fact to
$L^p_{loc}$ for some $p > 2$ which depends only on $K$. By H\"older's inequality, this stronger integrability
yields the distortion of area by quasiconformal maps. The formal statement of this phenomenon is as follows: if $F(0) = 0$ and $F(1) = 1$, then for all measurable subsets $E$ of the unit disc $\bD=\{z\in\mathbb{C}:|z|<1\}$ we have
\begin{equation}\label{gr}
 |F(E)|\leq C|E|^\kappa,
\end{equation}
where $|\cdot|$ denotes the Lebesgue measure on the plane, and the constants $C$ and $\kappa$ depend only on $K$. 
Gehring and Reich conjectured in \cite{GR} that the least possible $\kappa$ for which \eqref{gr} holds equals  $1/K$. This conjecture was open for many years, and was finally proved to be true by Astala \cite{Ast} in the early 90's. 

We will be interested in the following related estimate. The weak-type $(1,1)$ and $L^2$ boundedness of the Beurling-Ahlfors operator imply the existence of some universal constants $c$ and $\alpha$ such that for any subset $E$ of the unit disc $\bD$,
\begin{equation}\label{gr2}
 \int_{\bD}|B(\chi_E)(z)|\mbox{d}z\leq c|E|\ln\left(\frac{\alpha}{|E|}\right).
\end{equation}
Gehring and Reich \cite{GR} proved that their conjecture is strictly related to the optimal value of the constant $c$. Roughly speaking, the conjecture is equivalent to proving that the best $c$ in \eqref{gr2}  equals $1$. The inequality
$$  \int_E|B(\chi_E)(z)|\mbox{d}z\leq |E|,\qquad E\subseteq \bD,$$
follows immediately from the unitary property of $B$ and H\"older's inequality and  thus, the main difficulty while studying \eqref{gr2} lies in controlling the size of the integral over $\bD\setminus E$. In his celebrated paper {\it ``Area distortion of quasiconformal mappings"},  Astala \cite[Corollary 1.7]{Ast} proved \eqref{gr2} with $c=1$ and some constant $\alpha\geq 1$.   In \cite{EH}, Eremenko and Hamilton proved inequality \eqref{gr2} with the best constant $\alpha$ as well. That is, they show that 
\begin{equation}\label{gr3}
 \int_{\bD\setminus E}|B(\chi_E)(z)|\mbox{d}z\leq |E|\ln\left(\frac{\pi}{|E|}\right)
\end{equation}
holds for all Borel subsets $E$ of $\bD$ and that $\pi$ is the best possible. This gives \eqref{gr2}  with $c=1$ and $\alpha=e\pi$. For more (much, much more) on the operator $B$ and its connections to quasiconformal mappings in the plane, see Astala, Iwaniec and Martin \cite{AstIwaMar}, and specially $\S 14.6$ for above inequalities.

The objective of this paper is to study the above estimates in a probabilistic context and apply these to a wider class of Fourier multipliers. As evidenced in many papers (see e.g. \cite{BMH}, \cite{BW}, \cite{BorJanVol}, \cite{GMS}), martingale methods play a crucial role in the analysis of various inequalities for the Beurling-Ahlfors operator. To illustrate this fruitful connection, let us recall the celebrated conjecture of Iwaniec \cite{Iw}, which states that
$$ ||B||_{L^p(\mathbb{C})\to L^p(\mathbb{C})}=p^*-1,$$
where $p^*=\max\{p,p/(p-1)\}$. 
The first step towards the conjecture was made by Ba\~nuelos and Wang \cite{BW}, who proved that the norm is bounded by $4(p^*-1)$, using the representation of $B$ in terms of Poisson martingales. Then Nazarov and Volberg \cite{NV} used Bellman function techniques to improved the bound to $2(p^*-1)$ and Ba\~nuelos and Mend\'ez-Hernandez \cite{BMH} showed that this can also be done with the techniques of \cite{BW} by replacing the Poisson martingales by heat martingales. The more recent results of Ba\~nuelos and Janakiraman \cite{BJ} and Boros, Janakiraman and Volberg \cite{BorJanVol}, which provide the most tight bounds thus far, are based on ``fine-tuning'' the  heat martingale methods from \cite{BMH}.

In this paper we present a probabilistic study of inequalities of the type  \eqref{gr3}  based on novel and more sophisticated estimates (compared to those used for $L^p$-bounds for $B$) for differentially subordinated  martingales. This will allow us to obtain a more general statement (see Theorem \ref{mainthm} below) for a much wider class of Fourier multipliers. In particular, our reasoning will lead us to a certain sharp version of the Eremenko-Hamilton inequality for the real and imaginary part of the Beurling-Ahlfors operator.   The results also hold for other multipliers, for example those  of Laplace transform-type, which are not in the form given below and even for more general singular integral operators which are not of convolution-type but which arise from conditional expectations of martingale transforms. 
Furthermore, our approach will also give  a sharp martingale analogue of 
the result of Eremenko and Hamilton, stated below in  Corollary \ref{EreHamAnalo}.   

Let us now give more details on the above statements. We will work with the following class of symbols, introduced by Ba\~nuelos and Bogdan in \cite{BB}. 
Let $\nu$ be a L\'evy measure on $\R^d$.  That is, $\nu$ is a nonnegative Borel measure on $\R^d$ with $\nu(\{0\})=0$ and
$$ \int_{\R^d}\min\{|x|^2,1\}\nu(\mbox{d}x)<\infty.$$
Assume further that $\mu$ is a finite nonnegative Borel measure on the unit sphere $\mathbb{S}$ of $\R^d$ and fix two Borel functions $\phi$ on $\R^d$ and $\psi$ on $\mathbb{S}$ which take values in the unit ball of $\mathbb{C}$. We define the associated multiplier $m=m_{\phi,\psi,\mu,\nu}$ on $\R^d$ by
\begin{equation}\label{defm}
 m(\xi)=\frac{\frac{1}{2}\int_\mathbb{S} \langle \xi,\theta\rangle^2\psi(\theta)\mu(\mbox{d}\theta)+\int_{\R^d}[1-\cos\langle \xi,x\rangle ]\phi(x)\nu(\mbox{d}x)}{\frac{1}{2}\int_\mathbb{S} \langle \xi,\theta\rangle^2\mu(\mbox{d}\theta)+\int_{\R^d}[1-\cos\langle \xi,x\rangle]\nu(\mbox{d}x)}
 \end{equation}
if the denominator is not $0$, and $m(\xi)=0$ otherwise. Here $\langle\cdot,\cdot\rangle$ stands for the scalar product on $\R^d$. 
Note that  
\begin{eqnarray}\label{levymult}
m\left(\xi\right)&=&
\frac{\frac{1}{2}\langle{\mathcal{A}}\xi, \xi\rangle+
      \int_{\bR^d} [1-\cos \langle \xi,  x\rangle
     ]\phi\left(x\right)\nu(\mbox{d}x) 
     }
     {\frac{1}{2}\langle{\mathcal{B}}\xi, \xi\rangle+
      \int_{\bR^d} [1-\cos \langle\xi, x\rangle]\nu(\mbox{d}x) 
     },
\end{eqnarray}
where
\begin{eqnarray*}
 {\mathcal{A}} = \left( \int_{\uS} \psi\left(\theta\right)
  \theta_{i}\theta_{j}\, \mu(\mbox{d}\theta)
 \right)_{i,j=1 \ldots d} \quad \mbox{and} \quad
 {\mathcal{B}} = \left( \int_{\uS}  \theta_{i}\theta_{j}\, \mu(\mbox{d}\theta) \right)_{i,j=1 \ldots d}
\end{eqnarray*}
are $d\times d$ symmetric matrices and $\mathcal{B}$ is
non-negative definite.   We observe that $\|m(\xi)\|_{\infty}\leq 1$ and denote by $T_m$ the Fourier multiplier operator   with symbol $m$ given by $\widehat{T_{m}f}(\xi)=m(\xi)\hat{f}(\xi)$.  We  call these {\it L\'evy multipliers}. Via stochastic integration and the L\'evy-Khintchine formula, which is already apparent in \eqref{levymult} (as the denominator is the real part of the symbol of a L\'evy process), the operators $T_m$  can be represented as projections of martingale transforms. We refer the reader to Ba\~nuelos, Bielaszewski and Bogdan \cite{BBB} and Applebaum and Ba\~nuelos \cite[\S5]{AppBan} for details.   Combining this representation with Burkholder's martingale inequalities (cf. \cite{B0}, \cite{W}), Ba\~nuelos and Bogdan \cite{BB} and Ba\~nuelos, Bielaszewski and Bogdan \cite{BBB} obtained the following $L^p$ bound, strictly related to Iwaniec's conjecture.

\begin{theorem}\label{LPP}
Let $1<p<\infty$ and let $m=m_{\phi,\psi,\mu,\nu}$ be given by \eqref{defm}. Then for any $f\in L^p(\R^d)$ we have
\begin{equation}\label{Lp}
 ||T_mf||_{L^p(\R^d)}\leq (p^*-1)||f||_{L^p(\R^d)},
\end{equation}
where $p^*=\max\{p,p/(p-1)\}$. The bound $(p^*-1)$ cannot be improved. 
\end{theorem}

The class of {\it L\'evy multipliers} contains many interesting examples, including second order Riesz transforms and various Marcinkiewicz-type multipliers. When $d=2$, they include the real part $$\Re{B}=R_2^2-R_1^2$$ and imaginary part $$\Im{B}=-2R_2R_1$$ of the Beurling-Ahlfors operator $B$, where $R_1$ and $R_2$ are the Riesz transforms on $\R^2$. The fact that the  constant $(p^*-1)$ cannot be replaced by a smaller one follows from a result of  Geiss, Montgomery-Smith and Saksman \cite{GMS} (see also \cite{BO}) concerning the $L^p$ norms of $\Re{B}$ and $\Im{B}$, once one observes that these symbols belong to the class given by \eqref{defm} (see \cite[Example 4.2.1]{Ban1}).  For related logarithmic estimates  that can be obtained with the use of martingale methods, we refer the reader to Os\c ekowski \cite{O3} and \cite{O4}.  The main result of this paper is the following.

\begin{theorem}\label{mainthm} Let $m=m_{\phi,\psi,\mu,\nu}$ be a symbol given by \eqref{defm} and consider the Fourier multiplier $T_m$.  
If $D$, $E$ are measurable subsets of $\R^d$ with $E\subseteq D$ and $|D|<\infty$, then
\begin{equation}\label{mainin}
 \int_{D\setminus E} |T_{m}\chi_E(x)|\mbox{d}x\leq \begin{cases}
 |E|+|E|\ln\left(\frac{|D|}{2|E|}\right), & \mbox{if }|E|<|D|/2,\\
 |D\setminus E|+\frac{1}{2}|D \setminus E|\ln \left(\frac{|E|}{|D\setminus E|}\right), & \mbox{if }|E|\geq |D|/2.
\end{cases}
\end{equation}
\end{theorem}

From \eqref{mainin} and the elementary inequality 
$$
(1-a)+\frac{(1-a)}{2}\ln\left(\frac{1-a}{a}\right)\leq a+a\ln\left({\frac{1}{2a}}\right)=a\ln\left({\frac{e}{2a}}\right)
$$
valid for all $\frac{1}{2}\leq a\leq 1$, 
we obtain the following estimate  in the form of   \eqref{gr3}.  
\begin{corr}\label{maincor} Let $m$, $T_m$, $D$ and $E$ be as in Theorem \ref{mainthm}. Then 
\begin{equation}\label{mainin1}
\int_{D\setminus E} |T_{m}\chi_E(x)|\mbox{d}x\leq |E|\ln\left(\frac{e|D|}{2|E|}\right). 
\end{equation}
\end{corr}
It is important to note here, however, that if $E$ is large relative to $D$, that is if $|D\setminus E|$ is small, then the left hand side of \eqref{mainin}  is small and this is also reflected on the right hand side of that inequality.  On the other hand, the inequality \eqref{mainin1} alone does not capture this  feature.  Thus the probabilistic techniques yielding  \eqref{mainin} are rather fine-tuned and give behavior for our multipliers which follows that of the sharp inequality \eqref{gr3} for the Beurling-Ahlfors operator. That is, with the sharp constant $\pi$,  the right hand side of inequality \eqref{gr3} goes to zero as $|E|$ goes to either zero or $\pi$ which is the behavior of its left hand side. This behavior is also captured in the case of inequality \eqref{mainin}.   

As we have already mentioned above, the real and imaginary parts of the Beurling-Ahlfors operator belong to the class \eqref{defm}. Therefore, we obtain the following inequality for any measurable subset $E$ of the unit disc $\bD\subset \mathbb{C}$: 
\begin{equation*}
 \int_{\bD\setminus E}|\Re{B}(\chi_E)(z)|\mbox{d}z\leq 
\begin{cases}
\displaystyle |E|+|E|\ln\left(\frac{\pi}{2|E|}\right), & \mbox{if }|E|<\pi/2,\\
\displaystyle (\pi-|E|)+(\pi-|E|)\ln\left(\frac{|E|}{2(\pi-|E|)}\right), & \mbox{if }|E|\geq \pi/2,
\end{cases}
\end{equation*}
and  similarly with $\Re{B}$ replaced by $\Im{B}$ on the left hand side. 
In particular, 
\begin{equation}\label{real}
\int_{\bD\setminus E}|\Re{B}(\chi_E)(z)|\mbox{d}z\leq  |E|\ln\left(\frac{e\pi}{2|E|}\right)
\end{equation}
and 
\begin{equation}
 \int_{\bD\setminus E}|\Im{B}(\chi_E)(z)|\mbox{d}z \leq |E|\ln\left(\frac{e\pi}{2|E|}\right).
\end{equation} 
From this  we have the following inequality which is a non-sharp version of \eqref{gr3}.
\begin{corr}
For  any measurable subset $E$ of the unit disc $\bD\subset \mathbb{C}$ we have
\begin{equation*}
 \int_{\bD\setminus E}|B(\chi_E)(z)|\mbox{d}z\leq 
\begin{cases}
\displaystyle \sqrt{2}|E|+\sqrt{2}|E|\ln\left(\frac{\pi}{2|E|}\right), & \mbox{if }|E|<\pi/2,\\
\displaystyle \sqrt{2}(\pi-|E|)+\sqrt{2}(\pi-|E|)\ln\left(\frac{|E|}{2(\pi-|E|)}\right), & \mbox{if }|E|\geq \pi/2.
\end{cases}
\end{equation*}
In particular, 
\begin{equation}\label{sqrt2}
\int_{\bD\setminus E}|{B}(\chi_E)(z)|\mbox{d}z\leq  \sqrt{2}|E|\ln\left(\frac{e\pi}{2|E|}\right).   
\end{equation}
 \end{corr}
 
 Given the sharpness of the inequality \eqref{gr3} for the Beurling-Ahlfors operator and the sharpness of the martingale estimates used to obtain our  results, it is reasonable to conjecture that our inequalities are also sharp and that one should be able to show this by looking at the real or imaginary part of the Beurling-Ahlfors operator. This observation turned out to be successful in the study of \eqref{Lp} and, as we will see, this is the state of affairs  also in our setting. How can we construct appropriate extremal examples for $\Re{B}$ and $\Im{B}$? The first idea that comes to mind is to inspect carefully  the optimizers of \eqref{gr3}. Namely, for \eqref{gr3} one derives that
\begin{equation}\label{B-disc}
B(\chi_E)(z)=\frac{r^2}{z^2}\chi_{\bC\setminus E}(z),
\end{equation}
 where $E=\bD(0, r)$ is a disc centered at $0$ and radius $r$,
 and then explicitly computes the left hand side of \eqref{gr3} (see \cite[p.~386]{AstIwaMar}) and this yields equality. However, these calculations do not give the sharpness for our inequalities. Though the formula for $B(\chi_E)$ does yield the formula for $\Re{B}(\chi_E)$ and we can then compute the left hand side of \eqref{real}, we do not obtain the required equality at the end. Therefore, another approach is needed. We will make use of a much more complicated argument, which exploits the theory of laminates. It will allow us to establish the following result.

 \begin{theorem}\label{sharpness} For any   $\eta<1$ there are subsets $E\subset D\subset \bD$  such that 
\begin{equation}\label{sharp}
 \int_{D\setminus E} |(\Re{B})\chi_E(x)|\mbox{d}x\geq \eta\left[|E|\ln\left(\frac{e|D|}{2|E|}\right)\right].
 \end{equation}
 A similar statement holds for  $\Re{B}$ replaced by  $\Im{B}$.  
 \end{theorem}
 
Given the  inequalities \eqref{gr3}, \eqref{mainin} and \eqref{sharp}, several comments are in order. 
First, we observe  that  inequality \eqref{mainin} is the farthest point where differentially subordinate martingales can take us in the direction of \eqref{gr3}, when this is viewed purely as a two dimensional  case of more general \emph{real} harmonic analysis results on multipliers in $\bR^d$, $d\geq 2$.  This point of view completely ignores  the complex structure of $B$ on $\bC$;  for example, it ignores its  important property \eqref{intert}.  The results in Ba\~nuelos and Janakiraman \cite{BJ}, which have produced the best general bounds thus far on Iwaniec's conjecture, are an attempt to use, in a probabilistic  way, the complex structure of the operator $B$ without ``decoupling" it into its real and imaginary parts. The idea in \cite{BJ} is to take advantage of the ``conformal" structure, first noticed in \cite{BW},  of the martingales that arise from $B$  as well as the use of subordination.  (For more on several results on conformal martingales  motivated by the complex structure of the operator $B$, see \cite{BJ}, \cite{BO1}, \cite{BJV1} and \cite{Jan}.) A similar ``fine-tuning'' with the use of conformal martingales may also be successful for \eqref{sqrt2}. We believe the ideas in \cite{BJ} and the techniques in the current paper will likely lead to a better estimate than the $\sqrt{2}$ 
given in \eqref{sqrt2} but the construction of the necessary ``Burkholder" function is not at all clear at this point.

The second important remark we wish to make concerns the appearance of the three sets $D$, $\bD$ and $E$ in the formulation of Theorem \ref{sharpness}. The use of three sets is necessary: one cannot take $D=\bD$, since this would then clearly contradict \eqref{gr3}. This reveals a very interesting phenomenon: using martingale methods, we have proved that the geometry of the set $\bD$ plays a crucial role in the sharpness of \eqref{gr3} in the sense  that $\bD$ cannot be replaced there by an arbitrary subset of $\bC$ of finite measure. On the other hand, this should be confronted with the well-known fact that there is some ambiguity in the choice of $\bD$. Namely, as shown by Eremenko and Hamilton in \cite{EH}, \eqref{gr3} still holds true if this set is an arbitrary compact set of transfinite diameter $1$.

This paper is organized as follows.  \S2 contains the main probabilistic results: we prove an appropriate stochastic version of \eqref{mainin} there. In \S3 we show how to deduce the result for our L\'evy  multipliers from the martingale inequalities. Theorem \ref{sharpness} is proved in \S4. 

\section{martingale inequalities}
Assume that $(\Omega,\mathcal{F},\mathbb{P})$ is a complete probability space, equipped with $(\F_t)_{t\geq 0}$, a nondecreasing family of sub-$\sigma$-fields of $\F$, such that $\F_0$ contains all the events of probability $0$. Let $X$, $Y$ be two adapted c\'adl\'ag martingales, i.e., with right-continuous trajectories that have limits from the left. We assume further that $X$ takes values in the interval $[0,1]$, while $Y$ is $\mathcal{H}$-valued; here $\mathcal{H}$ denotes the separable Hilbert space, which may and will be assumed to be equal to $\ell_2$. The symbols $[X,X]$ and $[Y,Y]$ stand for the square brackets of $X$ and $Y$, respectively; see e.g. Dellacherie and Meyer \cite{DM} for the definition in the real-valued case, and extend the notion to the vector setting by $[Y,Y]=\sum_{k=1}^\infty [Y^k,Y^k]$, where $Y^k$ is the $k$-th coordinate of $Y$. Following Ba\~nuelos and Wang \cite{BW} and Wang \cite{W}, we say that $Y$ is \emph{differentially subordinate} to $X$, if the process $([X,X]_t-[Y,Y]_t)_{t\geq 0}$ is nonnegative and nondecreasing as a function of $t$. For example, let $f=(f_n)_{n\geq 0}$, $g=(g_n)_{n\geq 0}$ be a pair of adapted discrete-time martingales and let us treat them as continuous-time processes (via $X_t=f_{\lfloor t\rfloor}$, $Y_t=g_{\lfloor t\rfloor}$, $t\geq 0$). Then the above domination amounts to saying that $|dg_n|\leq |df_n|$ for all $n$, which is the original definition of differential subordination, due to Burkholder \cite{B0}. Here $(df_n)_{n\geq 0}$, $(dg_n)_{n\geq 0}$ stand for the difference sequences of $f$ and $g$, given by $df_0=f_0$, $df_n=f_n-f_{n-1}$ ($n\geq 1$), and similarly for $dg$.

We turn to the main result of this section. Let
\begin{equation}\label{defC}
 C(\lambda)=\begin{cases}
\lambda-\ln(2\lambda) & \mbox{if }0<\lambda\leq 1/2,\\
\frac{1}{2}e^{1-2\lambda} & \mbox{if }\lambda\geq 1/2.
\end{cases}
\end{equation}

\begin{theorem}\label{mainth}
Assume that $X$, $Y$ are martingales taking values in $[0,1]$ and $\mathcal{H}$, respectively. If $Y$ is differentially subordinate to $X$ and satisfies $Y_0=0$, then for any $\lambda> 0$ we have
\begin{equation}\label{martin}
\sup_{t\geq 0}\E \Big[(|Y_t|-\lambda)_+(1-X_t)\Big]\leq C(\lambda)\sup_{t\geq 0}\E [X_t].
\end{equation}
The inequality is sharp, even in the discrete-time setting: for any $\e>0$ there is a martingale $f$ taking values in $[0,1]$ and a real-valued martingale $g$ which is differentially subordinate to $f$, satisfying $g_0=0$ and 
\begin{equation}\label{sharp0}
\sup_{n\geq 0}\E \Big[(|g_n|-\lambda)_+(1-f_n)\Big]>\left(C(\lambda)-\e\right)\E f_0.
\end{equation}
\end{theorem}

We observe that since the martingale $X$ takes values in $[0, 1]$ we in fact have that $X_t=E(X|\mathcal{F}_t)$ where $X$ is a random variable with values in $[0, 1]$ and hence $\sup_{t\geq 0}\E [X_t]=\|X\|_1=X_0$. 
Let $D=[0,1]\times \mathcal{H}$. The proof of the inequality \eqref{martin} will be based on Burkholder's method: we will deduce the validity of the estimate from the existence of a certain special function, satisfying appropriate majorization and concavity. See \cite{B1} or \cite{O0} for the detailed description of the technique. Actually, we will exploit the following statement, which is a slight modification of the results of Wang \cite{W} (see Proposition~2 there).

\begin{lemma}\label{mainlemma}
Let $U:D\to \R$ be a continuous function which is of class $C^1$ in the interior of $D$ and of class $C^2$ on $D_i$, where $D_1$, $D_2$, $\ldots,$ $D_m$ are open subsets of $D$ such that $\overline{D_1}\cup \overline{D_2}\cup \ldots \cup \overline{D_m}=D$. Assume in addition that there is a Borel function $c:D_1\cup D_2\cup \ldots \cup D_m\to [0,\infty)$ satisfying
\begin{equation}\label{assumption1}
 \sup_{(x,y)\in (D_1\cup D_2\cup \ldots \cup D_m)\cap ([r,1-r]\times \mathcal{H})}c(x,y)<\infty \qquad \mbox{for all }0<r<1/2
\end{equation}
and such that for all $(x,y)\in D_1\cup D_2\cup \ldots \cup D_m$ and all $h\in \R$, $k\in \mathcal{H}$,
\begin{equation}\label{assumption2}
 U_{xx}(x,y)h^2+2\langle U_{xy}(x,y)h,k\rangle+\langle U_{yy}(x,y)k,k\rangle\leq -c(x,y)(|h|^2-|k|^2).
\end{equation}
Let $X$ be a martingale taking values in $[0,1]$ and let $Y$ be a martingale taking values in $\mathcal{H}$. If $Y$ is differentially subordinate to $X$, then for any $0<r<1/2$ and any $t\geq 0$ there is a nondecreasing sequence $(\tau_n)_{n\geq 0}$ of stopping times converging to $\infty$ almost surely, such that
\begin{equation}\label{assertion}
 \E U((1-r)X_{\tau_n\wedge t}+r,(1-r)Y_{\tau_n\wedge t})\leq \E U((1-r)X_0+r,(1-r)Y_0).
\end{equation}
\end{lemma}

Now we will introduce the special functions $U^\lambda:D\to \R$ corresponding to the inequality \eqref{martin}. We start with the case  $\lambda<1/2$ and consider the following subsets of $D$:
\begin{align*}
D_1&=\{(x,y)\in D: x+|y|< 2\lambda\},\\
D_2&=\{(x,y)\in D: 2\lambda <x+|y|<1\},\\
D_3&=\{(x,y)\in D: x+|y|>1\}.
\end{align*}
Let $U^\lambda$ be the function given by
$$ U^\lambda(x,y)=\begin{cases}
(4\lambda)^{-1}(|y|^2-x^2)+(\lambda-\ln(2\lambda))x & \mbox{if }(x,y)\in D_1,\\
|y|-\lambda+\lambda x-x\ln(x+|y|) & \mbox{if }(x,y)\in D_2,\\
(1-x)(x+|y|-\lambda) & \mbox{if }(x,y)\in D_3.
\end{cases}$$
It is easy to see that $U^\lambda$ extends to a continuous function on the whole strip $D$. 
During the proof of the properties listed in Lemma \ref{mainlemma}, we will also need the following auxiliary function $c^\lambda$ given on $D_1\cup D_2\cup D_3$:
$$ c^\lambda(x,y)=\begin{cases}
(2\lambda)^{-1} &  \mbox{if }(x,y)\in D_1,\\
(x+|y|)^{-1} & \mbox{if }(x,y)\in D_2,\\
1 & \mbox{if }(x,y)\in D_3.
\end{cases}$$
Next, we turn to the case $\lambda\geq 1/2$. This time the special function is given by four different formulas on the following subsets of the strip  $[0,1]\times \mathcal{H}$:
\begin{align*}
D_0&=\big\{(x,y):|y|\leq \min\{x,1-x\}\big\},\\
D_1&=\{(x,y): 0\leq x\leq 1/2,\,x<|y|<x+\lambda-1/2\},\\
D_2&=\{(x,y):1/2<x\leq 1,\,1<x+|y|<\lambda+1/2\},\\
D_3&=\big([0,1]\times \mathcal{H}\big)\setminus (D_0\cup D_1\cup D_2).
\end{align*}
Let $U^\lambda :[0,1]\times \mathcal{H}\to\R$ be a continuous function given by
$$ U^\lambda(x,y)=\begin{cases}
\frac{1}{2}\exp(1-2\lambda)(|y|^2-x^2+x) & \mbox{if }(x,y)\in D_0,\\
\frac{1}{2}x\exp\big(2|y|-2x-2\lambda+1\big) & \mbox{if }(x,y)\in D_1,\\
\frac{1}{2}(1-x)\exp\big(2|y|+2x-2\lambda-1\big) & \mbox{if }(x,y)\in D_2,\\
\frac{1}{2}\big[(|y|-\lambda+1/2)^2-x^2+x\big] & \mbox{if }(x,y)\in D_3.
\end{cases}$$
As in the case $\lambda<1/2$, we will also exploit the auxiliary function $c^\lambda:[0,1]\times \mathcal{H}\to [0,\infty)$. It is given by
$$ \quad \,\,\,c^\lambda(x,y)=\begin{cases}
\exp(1-2\lambda) & \mbox{if }(x,y)\in D_0,\\
\exp(2|y|-2x-2\lambda+1)\qquad \qquad  & \mbox{if }(x,y)\in D_1,\\
\exp(2|y|+2x-2\lambda-1) & \mbox{if }(x,y)\in D_2,\\
1 & \mbox{if }(x,y)\in D_3.
\end{cases}$$

We turn to the analysis of the above objects.
\begin{lemma}\label{propo}
For any $\lambda>0$, the functions $U^\lambda$ and $c^\lambda$ satisfy the assumptions of Lemma~\ref{mainlemma}.
\end{lemma}
\begin{proof}
The fact that $U^\lambda$ is of class $C^1$ in the interior of $D$ is straightforward and reduces to the tedious verification that its partial derivatives match at the common boundaries of the sets $D_i$. We leave the details to the reader. It is also evident that if $x$ is bounded away from $0$ and $1$, then the function $c^\lambda$ is uniformly bounded: this gives \eqref{assumption1}. The main technical difficulty lies in proving the inequality \eqref{assumption2}. Let us start with the case $\lambda<1/2$. If $(x,y)\in D_1$, then \eqref{assumption1}  is actually an equality. A little computation shows that if $(x,y)\in D_2$, then the left-hand side of \eqref{assumption2} equals
$$ -c^\lambda(x,y)(h^2-|k|^2)-|y|(x+|y|)^{-2}(h+\langle y',k\rangle)^2$$
(here and below, we  use the notation $y'=y/|y|$), so the bound holds true. Finally, if $(x,y)\in D_3$, then
\begin{align*}
U_{xx}(x,y)h^2&+2\langle U_{xy}(x,y)h,k\rangle+\langle U_{yy}(x,y)k,k\rangle\\
&= -2|h|^2-2h\langle y',k\rangle+\left(|k|^2-\langle y',k\rangle^2\right)\cdot \frac{1-|x|}{|y|}\\
&\leq \left(-|h|^2-2h\langle y',k\rangle-\langle y',k\rangle^2\right)-|h|^2+|k|^2\\
&\leq -|h|^2+|k|^2=-c^\lambda(x,y)(|h|^2-|k|^2). 
\end{align*}

We turn to the case $\lambda\geq 1/2$. The inequality \eqref{assumption2} is obvious for $D_0$; in fact, we get equality here. It is also easy to show the bound on $D_3$. Indeed, on this set we have
$$ 2U^\lambda(x,y)=|y|^2-x^2+x-2|y|(\lambda-1/2)+(\lambda-1/2)^2.$$
If the term $2|y|(\lambda-1/2)$ were absent, we would have equality in \eqref{assumption2}; since $\lambda\geq 1/2$, the function $(x,y)\mapsto 2|y|(\lambda-1/2)$ is convex and hence the desired bound is preserved. Next, we turn to the case when $(x,y)\in D_1$. Then it can be computed that the left-hand side of \eqref{assumption2} is equal to 
$ c^\lambda(x,y)(h^2-|k|^2)+I+II$, where
\begin{align*}
I&=\exp\big(2|y|-2x-2\lambda+1\big)(2x-1)\left(\frac{y\cdot k}{|y|}-h\right)^2\\
II&=\exp\big(2|y|-2x-2\lambda+1\big)(x/|y|-1)\left(|k|^2-\frac{(y\cdot k)^2}{|y|^2}\right).
\end{align*}
By the definition of $D_1$, we have $x\leq 1/2$ and $x\leq |y|$, which implies that both $I$ and $II$ are nonpositive; thus \eqref{assumption2} follows. Finally, to show the bound for $D_2$, we observe that $U^\lambda(x,y)=U^\lambda(1-x,y)$ for all $x,\,y$, so the inequality follows at once from the calculations for $D_1$.
\end{proof}

We will also need the following additional properties of $U^\lambda$. Recall the function $C$, given by \eqref{defC}.

\begin{lemma}
(i) For any $x\in [0,1]$ and $\lambda>0$ we have
\begin{equation}\label{init}
U^\lambda(x,0)\leq C(\lambda)x.
\end{equation}

(ii) For any $(x,y)\in D$ and $\lambda>0$ we have
\begin{equation}\label{maj}
U^\lambda(x,y)\geq (y-\lambda)_+(1-x).
\end{equation}
\end{lemma}
\begin{proof}
Note that the conditions listed in Lemma \ref{mainlemma}, which have been proved above, yield the following property of $U^\lambda$: for any fixed  $y\in\mathcal{H}$, the function $x\mapsto U^\lambda(x,y)$ is concave. Having observed this, \eqref{init} follows at once, because it is equivalent to
$$ U^\lambda(x,0)\leq U^\lambda(0,0)+U^\lambda_x(0+,0)x.$$
Furthermore, the concavity implies that it is enough to prove \eqref{maj} for $x\in \{0,1\}$, since for any fixed $y$, the right-hand side of \eqref{maj} is linear in $x$. Suppose first that $x=0$ and $\lambda<1/2$. If $|y|\leq 2\lambda$, the majorization is equivalent to $(|y|-2\lambda)^2\geq 0$, which is of course true; if $|y|>2\lambda$, then both sides are equal. Assume next that $x=0$ and $\lambda\geq 1/2$. If $|y|\leq \lambda-1/2$, then we get equality; if $|y|>\lambda-1/2$, the inequality \eqref{maj} can be rewritten as $(|y|-\lambda-1/2)^2\geq 0$, which holds true. If $x=1$, then the majorization is equivalent to $U^\lambda(1,|y|)\geq 0$, which is evident for all choices of $\lambda$.
\end{proof}

We turn to the proof of Theorem \ref{mainth}.

\begin{proof}[Proof of \eqref{martin}] 
Combining Lemma \ref{mainlemma} and Lemma \ref{propo}, we obtain that for any $0<r<1/2$, any $t\geq 0$ and appropriate sequence $(\tau_n)_{n\geq 0}$ of stopping times,
\begin{align*}
  \E U^\lambda((1-r)X_{\tau_n\wedge t}+r,(1-r)Y_{\tau_n\wedge t})&\leq \E U^\lambda((1-r)X_0+r,(1-r)Y_0)\\
  &=\E U^\lambda((1-r)X_0+r,0).
 \end{align*}
By \eqref{maj}, this implies
$$  (1-r)\E\Big[ \big((1-r)|Y_{\tau_n\wedge t}|-\lambda\big)_+(1-X_{\tau_n\wedge t})\Big]\leq \E U^\lambda((1-r)X_0+r,0).$$
Now we let $n\to \infty$ and then $r\downarrow 0$ to obtain, in the light of Fatou's lemma,
$$ \E (|Y_t|-\lambda)_+(1-X_t)\leq \E U^\lambda(X_0,0).$$
It remains to use \eqref{init} and take the supremum over $t$. The inequality \eqref{martin} is established.
\end{proof}

\begin{proof}[Sharpness of \eqref{martin}]
Now we will construct discrete-time martingales showing that the constant $C(\lambda)$ cannot be replaced by a smaller number. Let us start with the case $\lambda<1/2$. Let $\kappa$ be a small number belonging to $[0,1]$ and let $N$ be a large positive integer. Set $\delta=(1-2\lambda)/(2N)$ and consider the Markov martingale $(f,g)$, whose distribution is uniquely determined by the following requirements.
\begin{itemize}
\item[(i)] We have $(f_0,g_0)\equiv (\kappa,0)$.
\item[(ii)] The state $(\kappa,0)$ leads to $(\kappa/2,-\kappa/2)$ or to $(\lambda+\kappa/2,\lambda-\kappa/2)$.
\item[(iii)] The state $(\kappa/2,-\kappa/2)$ leads to $(0,0)$ or to $(\lambda,-\lambda)$.
\item[(iv)] The state $(\lambda+\kappa/2,\lambda-\kappa/2)$ leads to $(0,2\lambda)$ or to $(2\lambda,0)$.
\item[(v)] The state $(\lambda,-\lambda)$ leads to $(0,-2\lambda)$ or to $(2\lambda,0)$.
\item[(vi)] The state of the form $(2\lambda+2k\delta,0)$ ($k=0,\,1,\,2,\,\ldots,\, N-1$) leads to $(0,2\lambda+2k\delta)$ or to $(2\lambda+2k\delta+\delta,-\delta)$.
\item[(vii)] The state of the form $(2\lambda+2k\delta+\delta,0)$ ($k=0,\,1,\,2,\,\ldots,\, N-1$) leads to $(0,-2\lambda-2k\delta-2\delta)$ or to $(2\lambda+2k\delta+2\delta,0)$.
\item[(viii)] All the states not mentioned above are absorbing.
\end{itemize}
Note that there is no need to specify the transition probabilities, they are uniquely determined by the requirement that $(f,g)$ is a martingale. Directly from the above definition, we infer that $0=|g_0|\leq f_0$ and for each $n$ we have $|dg_n|=|df_n|$; therefore, $g$ is differentially subordinate to $f$. We easily see that $(f,g)$ is a finite martingale: let $(f_\infty,g_\infty)$ denote its terminal variable. Clearly, we have $f_\infty \in\{0,1\}$, so
$$ \E \Big[(|g_\infty|-\lambda)_+(1-f_\infty)\Big]=\E (|g_\infty|-\lambda)_+.
 $$
We easily see from the conditions (i)-(viii) above that the variable $|g_\infty|$ takes values $0$, $2\lambda$, $2\lambda+2\delta$, $2\lambda+4\delta$, $\ldots$, $1$. Using (ii)-(v), we compute that 
\begin{align*}
 \mathbb{P}(|g_\infty|=2\lambda)&=\frac{\kappa}{2\lambda}\cdot\frac{\lambda-\kappa/2}{2\lambda}+\frac{2\lambda-\kappa}{2\lambda}\cdot\frac{\kappa}{2\lambda}\cdot \frac{1}{2}+\frac{\kappa}{2\lambda}\cdot \frac{\delta}{2\lambda+\delta}\\
&=\frac{\kappa(2\lambda-\kappa)}{4\lambda^2}+\frac{\kappa}{2\lambda}\cdot \frac{\delta}{2\lambda+\delta}.
 \end{align*}
 Indeed, $\frac{\kappa}{2\lambda}\cdot\frac{\lambda-\kappa/2}{2\lambda}$ is the probability that $f$ goes to $\lambda+\kappa/2$ and then jumps to $0$; the second summand corresponds to the case when $f$ goes to $\kappa/2$ in the first step, then to $\lambda$ and finally to $0$; the third term $\frac{\kappa}{2\lambda}\cdot \frac{\delta}{2\lambda+\delta}$ comes from the following possibility: $f$ might get to the point $2\lambda$ after a few steps; this occurs with probability $\kappa/(2\lambda)$; if it is so, $f$ may jump to $0$, which happens with probability $\delta/(2\lambda+\delta)$ (see (vi)).

The next step is to prove that for $k=0,\,1,\,2,\,\ldots,\,N$,
$$ \mathbb{P}(f\,\mbox{ ever visits }\,2\lambda+2k\delta)=\frac{\kappa}{2\lambda+2k\delta}.$$
The case $k=0$ has already appeared in the above considerations, the general case follows from an easy induction and the requirements (vi) and (vii). In consequence, by the further exploitation of these two conditions, we get that for $k=1,\,2,\,\ldots,\,N-1$, the event $\{|g_\infty|=2k\delta\}$ is the disjoint union of the following two: either 
$f$ visits $2\lambda+2(k-1)\delta$ after several steps, then goes to $2\lambda+2(k-1)\delta+\delta$ and then to $0$, or $f$ visits $2\lambda+2(k-1)\delta$ after several steps, then goes to $2\lambda+2(k-1)\delta+\delta$, then to $2\lambda+2k\delta$ and then to $0$. Computing the corresponding probabilities, we see that
\begin{align*} 
\mathbb{P}(|g_\infty|=2\lambda+2k\delta)&=\frac{\kappa}{2\lambda+2(k-1)\delta}\cdot \frac{2\lambda+2(k-1)\delta}{2\lambda+(2k-1)\delta} \times \\
&\quad \times  \left(\frac{\delta}{2\lambda+2k\delta}+\frac{2\lambda+(2k-1)\delta}{2\lambda+2k\delta}\cdot \frac{\delta}{2\lambda+2k\delta+\delta}\right)\\
&=\frac{2\kappa\delta}{(2\lambda+(2k-1)\delta)(2\lambda+(2k+1)\delta)}\\
&\geq \frac{2\kappa \delta}{(2\lambda+2k\delta)^2}.
\end{align*}
The probability $\mathbb{P}(|g_\infty|=1)$ can be derived similarly, but actually we will not need this. Namely, we can write
\begin{align*}
\frac{\E \big[(|g_\infty|-\lambda)_+(1-f_\infty)\big]}{\E f_0}=\frac{\E (|g_\infty|-\lambda)_+}{\kappa}&\geq \frac{2\lambda-\kappa}{4\lambda}+\sum_{k=1}^{N-1} \frac{2\delta(\lambda+2k\delta)}{(2\lambda+2k\delta)^2}.
\end{align*}
Now if $\delta$ is appropriately small, then the latter expression  can be made arbitrarily close to
$$ \frac{2\lambda-\kappa}{4\lambda}+\int_\lambda^{1-\lambda} \frac{x}{(\lambda+x)^2}\mbox{d}x=-\frac{\kappa}{4}+\lambda-\ln(2\lambda)=-\frac{\kappa}{4}+C(\lambda).$$
Letting $\kappa\to 0$ we see that the constant $C(\lambda)$ is indeed optimal in \eqref{martin} for $\lambda<1/2$.

We turn to the case $\lambda\geq 1/2$. First, let us introduce the extremal sequence for $\lambda=1/2$. Fix a small $\kappa\in [0,1]$ and let $(f,g)$ satisfy
\begin{itemize}
\item[(i)] We have $(f_0,g_0)\equiv (\kappa,0)$.
\item[(ii)] The state $(x,0)$ leads to $(\kappa/2,-\kappa/2)$ or to $(1/2+\kappa/2,1/2-\kappa/2)$.
\item[(iii)] The state $(\kappa/2,-\kappa/2)$ leads to $(0,0)$ or to $(1/2,-1/2)$.
\item[(iv)] The state $(1/2+\kappa/2,1/2-\kappa/2)$ leads to $(0,1)$ or to $(1,0)$.
\item[(v)] The state $(1/2,-1/2)$ leads to $(0,-1)$ or to $(1,0)$.
\end{itemize} 
If $\lambda>1/2$, then fix a large positive integer $N$ and put $\delta=(\lambda-1/2)/(2N)$. Consider the martingale $(f,g)$ satisfying
\begin{itemize}
\item[(i)] We have $(f_0,g_0)\equiv (\kappa,0)$.
\item[(ii)] The state $(\kappa,0)$ leads to $(\kappa/2,-\kappa/2)$ or to $(1/2+\kappa/2,1/2-\kappa/2)$.
\item[(iii)] The state $(\kappa/2,-\kappa/2)$ leads to $(0,0)$ or to $(1/2,-1/2)$.
\item[(iv)] The state $(1/2+\kappa/2,1/2-\kappa/2)$ leads to $(1/2,1/2)$ or to $(1,0)$.
\item[(v)] The state of the form $(1/2,1/2+2k\delta)$, $k=0,\,1,\,2,\,\ldots,\,N-1$ leads to $(0,2k\delta)$ or to $(1/2+\delta,1/2+2k\delta+\delta)$. Symmetrically, the state of the form $(1/2,-1/2-2k\delta)$, $k=0,\,1,\,2,\,\ldots,\,N-1$ leads to $(0,-2k\delta)$ or to $(1/2+\delta,-1/2-2k\delta-\delta)$.
\item[(vi)] The state of the form $(1/2+\delta,1/2+2k\delta+\delta)$, $k=0,\,1,\,2,\,\ldots,\,N-1$ leads to $(1,2k\delta+2\delta)$ or to $(1/2,1/2+2k\delta+2\delta)$. Symmetrically, the state of the form $(1/2+\delta,-1/2-2k\delta-\delta)$, $k=0,\,1,\,2,\,\ldots,\,N-1$ leads to $(1,-2k\delta-2\delta)$ or to $(1/2,-1/2-2k\delta-2\delta)$.
\item[(vii)] The state $(1/2,\lambda)$ leads to $(0,\lambda+1/2)$ or to $(1,\lambda-1/2)$. Symmetrically, the state $(1/2,-\lambda)$ leads to $(0,-\lambda-1/2)$ or to $(1,-\lambda+1/2)$.
\item[(viii)] All the states not mentioned above are absorbing.
\end{itemize} 
The calculations involved in the analysis of the above processes are  similar to those in the case $\lambda<1/2$. We leave the necessary verification to the reader.
\end{proof} 

Finally, let us formulate  a corollary which will be important to us later.

\begin{corr}
Assume that $X$, $Y$ are martingales taking values in $[0,1]$ and $\mathcal{H}$, respectively. If $Y$ is differentially subordinate to $X$ and satisfies $Y_0=0$, then for any $\lambda> 0$, $t\geq 0$ and any $A\in \F$ we have
\begin{equation}\label{martin2}
\E |Y_t|(1-X_t)1_A\leq C(\lambda)\E X_0+\lambda\E(1-X_t)1_A .
\end{equation}
\end{corr}
\begin{proof}
Consider the following decomposition: $A=A^-\cup A^+$, where
$$ A^-=A\cap\{|Y_t|<\lambda\},\qquad A^+=A\cap\{|Y_t|\geq \lambda\}.$$
Clearly,
$$ \E (|Y_t|-\lambda)(1-X_t)1_{A^-}\leq 0$$
and
$$ \E(|Y_t|-\lambda)(1-X_t)1_{A^+}\leq \E (|Y_t|-\lambda)_+(1-X_t)\leq C(\lambda)\E X,$$
where in the last passage we have exploited \eqref{martin}. Adding the two estimates above, we get
$$ \E (|Y_t|-\lambda)(1-X_t)1_A\leq C(\lambda)\E X_0,$$
which is precisely the claim.
\end{proof}

Taking $A=\Omega$, we obtain 
\begin{equation}\label{martin3}
\E |Y_t|(1-X_t)\leq C(\lambda)\E X_0+\lambda\E(1-X_t) .
\end{equation}
Minimizing the right hand side with respect to $\lambda$ and and again recalling the notation $||X||_1=\sup_{t\geq 0}\E |X_t|=\E X_0$, we easily arrive at the following result which is a martingale analogue of the Eremenko-Hamilton inequality  \eqref{gr3}. 

\begin{corollary}\label{EreHamAnalo} Assume that $X$, $Y$ are martingales taking values in $[0,1]$ and $\mathcal{H}$, respectively. If $Y$ is differentially subordinate to $X$ and satisfies $Y_0=0$, then
\begin{equation*}
 ||Y(1-X)||_1\leq \begin{cases}
||X\|_1+||X||_1\ln \left(\frac{1}{2||X||_1}\right), & \mbox{if }\|X\|_1<1/2,\\
 (1-||X\|_1)+\frac{1}{2}(1-||X\|_1)\ln \left(\frac{1}{1-||X\|_1}\right), & \mbox{if }\|X\|_1\geq 1/2.
\end{cases}
\end{equation*}
In particular,
\begin{equation}\label{martin4}
||Y(1-X)||_1\leq ||X||_1\ln\left(\frac{e}{2||X||_1}\right),
\end{equation}
and the constant $e/2$ is best possible.  
\end{corollary}

The sharpness of the inequality \eqref{martin4} follows already from the examples above.  It is, however,  very easy to give an example in this case by considering the  discrete-time setting.  Suppose that $f_0=g_0\equiv 1/2$ and let $f_1-f_0=g_0-g_1$ be a Rademacher variable divided by $2$. Then $||f||_1=1/2$ and $\mathbb{P}(g(1-f)=1)=1/2$, so
$$ ||g(1-f)||_1\geq 1/2=||f||_1\log (e/2||f||_1).$$

\section{Proof of Theorem \ref{mainthm}} 
Now we will show how the martingale inequalities studied in the preceding section yield the corresponding bounds for Fourier multipliers. 
We start by recalling the martingale representation of the multipliers from the class \eqref{defm}. We follow here the description  in \cite{BB} and \cite{BBB} and refer the reader to those papers for full details.  An alternate description based on the semigroup of the L\'evy process and stochastic integration can be found in \cite[\S4]{Ban1} and \cite[\S5]{AppBan}.

Let $m$ be the multiplier as in \eqref{defm}, with the corresponding parameters  $\phi,\,\psi,\,\mu$ and $\nu$. Assume in addition that $\nu(\R^d)$ is finite and nonzero. Then for any $s<0$ there is a L\'evy process $(X_{s,t})_{t\in [s,0]}$ with $X_{s,s}\equiv 0$, for which Lemmas \ref{diflema} and \ref{diflema2} below hold true. To state these, we need some notation. \def\faaf{ and define $\tilde{\nu}=\nu/|\nu|$. Consider the independent random variables $T_{-1}$, $T_{-2}$, $\ldots$, $Z_{-1}$, $Z_{-2}$, $\ldots$ such that for each $n=-1,\,-2,\,\ldots$, $T_n$ has exponential distribution with parameter $|\nu|$ and $Z_n$ takes values in $\R^d$ and has $\tilde{\nu}$ as the distribution. Next, put  $S_n=-(T_{-1}+T_{-2}+\ldots+T_n)$ for $n=-1,\,-2,\ldots$ and let
$$ X_{s,t}=\sum_{s<S_j\leq t}Z_j,\qquad X_{s,t-}=\sum_{s<S_j< t}Z_j,\qquad \Delta X_{s,t}=X_{s,t}-X_{s,t-},$$
for $-\infty<s\leq t\leq 0$. }For a given $f\in L^\infty(\R^d)$, define the corresponding parabolic extension $\mathcal{U}_f$ to $(-\infty,0]\times \R^d$ by
$$ \mathcal{U}_f(s,x)=\E f(x+X_{s,0}).$$
Next, fix $x\in \R^d$, $s<0$ and let $f,\,\phi\in L^\infty(\R^d)$. We introduce the processes $F=(F^{x,s,f}_t)_{s\leq t\leq 0}$ and $G=(G^{x,s,f,\phi}_t)_{s\leq t\leq 0}$ by
\begin{equation}\label{defFG}
\begin{split}
 F_t&=\mathcal{U}_f(t,x+X_{s,t}),\\
 G_t&=\sum_{s<u\leq t}\big[(F_u-F_{u-}) \cdot \phi(X_{s,u}-X_{s,u-})\big]\\
&\quad -\int_s^t\int_{\R^d}\big[\mathcal{U}_f(v,x+X_{s,v-}+z)-\mathcal{U}_f(v,x+X_{s,v-}) \big]\phi(z)\nu(\mbox{d}z)\mbox{d}v.
\end{split}
\end{equation}
Now, fix $s<0$ and define the operator $\mathcal{S}=\mathcal{S}^{s,\phi,\nu}$ by the bilinear form
\begin{equation}\label{defS}
 \int_{\R^d}\mathcal{S}f(x)g(x)\mbox{d}x=\int_{\R^d}\E \big[G_0^{x,s,f,\phi}g(x+X_{s,0})\big]\mbox{d}x,
\end{equation}
where $f,\,g\in C_0^\infty(\R^d)$. We have the following facts, proved in \cite{BB} and \cite{BBB}. 
\begin{lemma}\label{diflema}
For any fixed $x,\,s,\,f,\,\phi$ as above, the processes $F^{x,s,f}$, $G^{x,s,f,\phi}$ are martingales with respect to $(\F_t)_{s\leq t\leq 0}=(\sigma(X_{s,t}:s\leq t))_{s\leq t\leq 0}$. Furthermore, if $||\phi||_\infty\leq 1$, then $G^{x,s,f,\phi}$ is differentially subordinate to $F^{x,s,f}$.
\end{lemma}
Let us stress here that $\phi$, and hence also $G$, are complex valued. Note that in addition, we have $G_0=0$. 
The aforementioned representation of Fourier multipliers in terms of L\'evy processes is as follows.

\begin{lemma}\label{diflema2}
Let $1<p<\infty$ and $d\geq 2$. The operator $\mathcal{S}^{s,\phi,\nu}$ is well defined and extends to a bounded operator on $L^p(\R^d)$, which can be expressed as a Fourier multiplier with the symbol
\begin{equation*}
\begin{split}
 M(\xi)&=M_{s,\phi,\nu}(\xi)\\
&=\left[1-\exp\left(2s\int_{\R^d}(1-\cos\langle \xi, z\rangle )\nu(\mbox{d}z)\right)\right]
\frac{\int_{\R^d}(1-\cos\langle \xi, z\rangle)\phi(z)\nu(\mbox{d}z)}{\int_{\R^d}(1-\cos\langle \xi, z\rangle)\nu(\mbox{d}z)}
\end{split}
\end{equation*}
if $\int_{\R^d}(1-\cos\langle \xi, z\rangle)\nu(\mbox{d}z)\neq 0$, and $M(\xi)=0$ otherwise.
\end{lemma}

We are ready to prove Theorem \ref{mainthm}.

\begin{proof}[Proof of \eqref{mainin}] Fix subsets $D$, $E$ of $\R^d$ as in the statement. 
We may and do assume that at least one of the measures $\mu$, $\nu$ is nonzero.
It is convenient to split the reasoning into two parts.

\smallskip

\emph{Step 1.} 
First we show the estimate for the multipliers of the form
\begin{equation}\label{defM}  M_{\phi,\nu}(\xi)=
\frac{\int_{\R^d}(1-\cos\langle \xi, z\rangle)\phi(z)\nu(\mbox{d}z)}{\int_{\R^d}(1-\cos\langle \xi, z\rangle)\nu(\mbox{d}z)}.
\end{equation}
Assume that $0<\nu(\R^d)<\infty$, so that the above machinery using L\'evy processes is applicable.  Fix $s<0$ and functions $f,\,g\in C_0^\infty(\R^d)$ such that $f$ takes values in $[0,1]$, while $g$ takes values in the unit ball of $\mathbb{C}$ and is supported on $D$. Of course, then the martingale $F^{x,s,f}$  takes values in $[0,1]$. By  Fubini's theorem and \eqref{martin2}, we have, for any $\lambda\geq 1/2$,
\begin{equation*}
\begin{split}
 & \left|\int_{\R^d}\E \big[G_0^{x,s,f,\phi}(1-F_0^{x,s,f})g(x+X_{s,0})\big]\mbox{d}x\right|\\
&\qquad \leq \int_{\R^d} \E |G_0^{x,s,f,\phi}|(1-F_0^{x,s,f})1_{\{x+X_{s,0}\in D\}}\mbox{d}x\\
 &\qquad \leq C(\lambda)\int_{\R^d}\E F_0^{x,s,f}\mbox{d}x+\lambda \int_{\R^d}\E (1-F_0^{x,s,f})1_{\{x+X_{s,0}\in D\}}\mbox{d}x\\
 &\qquad =C(\lambda)||f||_{L^1(\R^d)}+\lambda\int_{\R^d}\E (1-F_0^{x,s,f})1_{\{x+X_{s,0}\in D\}}\mbox{d}x.
\end{split}
\end{equation*}
Now we apply the definition of $\mathcal{S}$: note that $1-F_0^{x,s,f}=1-\mathcal{U}_f(0,x+X_{s,0})=1-f(x+X_{s,0})$ is a function of $x+X_{s,0}$, so \eqref{defS} with $\tilde{g}(x)=(1-f(x))g(x)$, we get
$$ \left|\int_{\R^d}\mathcal{S}f(x)(1-f(x))g(x)\mbox{d}x\right|\leq C(\lambda)||f||_{L^1(\R^d)}+\lambda\int_{\R^d}\E (1-F_0^{x,s,f})1_{\{x+X_{s,0}\in D\}}\mbox{d}x.$$
Taking the supremum over all $g$ as above, we obtain
\begin{equation}\label{od}
 \int_D |\mathcal{S}^{s,\phi,\nu}f(x)|(1-f(x))\mbox{d}x\leq  C(\lambda)||f||_{L^1(\R^d)}+\lambda\int_{\R^d}\E (1-F_0^{x,s,f})1_{\{x+X_{s,0}\in D\}}\mbox{d}x.
\end{equation}
Now if we let $s\to -\infty$, then $M_{s,\phi,\nu}$ converges pointwise to the multiplier $M_{\phi,\nu}$ given by \eqref{defM}.  
By Plancherel's theorem, $\mathcal{S}^{s,\phi,\nu}f \to T_{M_{\phi,\nu}}f$ in $L^2(\R^d)$ and hence there is a sequence $(s_n)_{n=1}^\infty$ converging to $-\infty$ such that  $\lim_{n\to\infty}\mathcal{S}^{s_n,\phi,\nu}f \to T_{M_{\phi,\nu}}f$ almost everywhere. Thus Fatou's lemma combined with \eqref{od} yields the bound
$$\int_D |T_{M_{\phi,\nu}}f(x)|(1-f(x))\mbox{d}x\leq C(\lambda)||f||_{L^1(\R^d)}+\lambda\int_{\R^d}\E (1-F_0^{x,s,f})1_{\{x+X_{s,0}\in D\}}\mbox{d}x.$$
Using standard approximation arguments, we see that the above bound holds true also for $f=\chi_E$. For such a choice of $f$, we get
$$ (1-F_0^{x,s,f})1_{\{x+X_{s,0}\in D\}}=1_{\{x+X_{s,0}\notin E\}}1_{\{x+X_{s,0}\in D\}}=1_{\{x+X_{s,0}\in D\setminus E\}},$$
so by Fubini's theorem, we obtain
\begin{equation}\label{wtmw}
 \int_{D\setminus E} |T_{M_{\phi,\nu}}\chi_E(x)|\mbox{d}x\leq C(\lambda)|E|+\lambda|D\setminus E|.
\end{equation}
 
\emph{Step 2.} Now we deduce the result for the general multipliers as in \eqref{defm} and drop the assumption $0<\nu(\R^d)<\infty$. For a given $\e>0$, define a L\'evy measure $\nu_\e$ in polar coordinates $(r,\theta)\in (0,\infty)\times \mathbb{S}$ by
$$ \nu_\e(\mbox{d}r\mbox{d}\theta)=\e^{-2}\delta_\e(\mbox{d}r)\mu(d\theta).$$
Here $\delta_\e$ denotes Dirac measure on $\{\e\}$. Next, consider a multiplier $M_{\e,\phi,\psi,\mu,\nu}$ as in \eqref{defM}, in which the L\'evy measure is $1_{\{|x|> \e\}}\nu+\nu_\e$ and the jump modulator is  given by $1_{\{|x|>\e\}}\phi(x)+1_{\{|x|=\e\}}\psi(x/|x|)$. Note that this L\'evy measure is finite and nonzero, at least for sufficiently small $\e$. If we let $\e\to 0$, we see that
\begin{equation*}
\begin{split}
 \int_{\R^d}[1-\cos\langle \xi,x\rangle]\psi(x/|x|)\nu_\e(\mbox{d}x)&=\int_{\mathbb{S}}\langle \xi,\theta\rangle^2\phi(\theta)\frac{1-\cos\langle \xi,\e\theta\rangle}{\langle \xi,\e\theta\rangle^2}\mu(d\theta)\\
 &\to \frac{1}{2}\int_{\mathbb{S}}\langle \xi,\theta\rangle^2\phi(\theta)\mu(\mbox{d}\theta)
 \end{split}
 \end{equation*}
 and, consequently, $M_{\e,\phi,\psi,\mu,\nu} \to m_{\phi,\psi,\mu,\nu}$ pointwise.  
 Thus \eqref{wtmw} yields 
\begin{equation}\label{wtmw2}
 \int_{D\setminus E} |T_{m}\chi_E(x)|\mbox{d}x\leq C(\lambda)|E|+\lambda|D\setminus E|.
\end{equation}
Indeed, using Plancherel's theorem as above, we see that there is a sequence $(\e_n)_{n\geq 1}$ converging to $0$ such that $T_{M_{\e_n,\phi,\psi,\mu,\nu}}\chi_E\to T_{m_{\phi,\psi,\mu,\nu}}\chi_E$ almost everywhere. It suffices to apply Fatou's lemma.

The next step is to optimize the right-hand side of \eqref{wtmw2} over $\lambda$. A straightforward analysis of the derivative shows that the following choices yield best bounds. If $|E|< |D|/2$, then taking $\lambda=|E|/|D|< 1/2$ gives
$$ \int_{D\setminus E} |T_{m}\chi_E(x)|\mbox{d}x\leq |E|+|E|\ln\left(\frac{|D|}{2|E|}\right). $$
On the other hand, if $|E|>|D|/2$, then we take $\lambda =\frac{1}{2}+\frac{1}{2}\ln\left(\frac{|E|}{|D\setminus E|}\right)$ and obtain
$$ \int_{D\setminus E} |T_{m}\chi_E(x)|\mbox{d}x\leq |D\setminus E|+\frac{1}{2}|D \setminus E|\ln\left( \frac{|E|}{|D\setminus E|}\right).$$
This completes the proof of the desired bound.
\end{proof}

\begin{remark}
The paper \cite{AppBan} contains the construction of a wider class of linear operators on $\R^d$, and more general Lie groups, for which the inequality  \eqref{mainin} holds true (which can be shown by modifying the above proof). These operators  are not necessarily Fourier multipliers, see \cite[Theorem 3.3 and Corollary 5.1]{AppBan}, but they do include  the class  multipliers of Laplace transform-type (again, both on $\R^d$ and on compact Lie groups) as well as second-order Riesz transforms on compact groups.  

For example, if $h_t(x)$ denotes the Gaussian kernel in $\bR^d$, the family of operators 
\begin{equation}\label{moregenral}
T_Af(x)=\int_{\bR^d}K(x, z)f(z) dz,
\end{equation}
where
$$
K(x, z)=\int_{0}^{\infty}\int_{\bR^d} \langle{A(y, t)\nabla h_t(x-y), \nabla h_t(y-z)}\rangle dy dt 
$$
under the assumption that $A(y, t)$ is $d\times d$ matrix valued function on $\bR^d\times [0, \infty)$ with
$$
\|A\|=\| |A(y,t)|\|_{L^{\infty} (\bR^d\times [0,+\infty))}\leq 1,
$$ $
 |A(y, t)|=\sup\{|A(y, t)\xi|;  \xi\in \bR^d, |\xi|=1\}$, are examples of the operators arising from \cite[Theorem 3.3 and Corollary 5.1]{AppBan} and for which the conclusion of Theorem \ref{mainthm} applies.  When $A=(a_{ij})$ is constant, $T_A=\sum_{i,j}a_{ij}R_iR_j$ where $R_i$, $R_j$ are the Riesz transforms.   When $A=a(t)I,$ $I$ the $d\times d$ identity matrix and $a$ a bounded function, we get  the operators of Laplace transform-type. Under the general assumptions above, $T_A$ is not necessarily convolution-type.  

\end{remark}

\section{Proof of Theorem \ref{sharpness}}

\def\diag{\operatorname*{diag}}
Our proof will actually show  that for any $\eta<1$ there are subsets $E$, $D$ of the unit disc $\bD$ such that $E\subset D$, $|E|$ is arbitrarily close to $0$ and
 \begin{equation}\label{sharp1}
 \int_{\bD\setminus E} |(\Re{B})\chi_E(x)|\mbox{d}x> \eta\left[|E|\ln\left(\frac{e|D|}{2|E|}\right)\right].
 \end{equation}
To accomplish this, we will make use of a combination of various analytic and probabilistic facts. It is convenient to split this section into a several separate parts. Our arguments here are very similar to those used by Os\c ekowski in \cite{O5} to prove the sharpness of restricted weak-type
estimates for the L\'evy multipliers.  For the convince of the reader, and in order to make this section as self contained as possible, we recall the preliminaries on laminates and their connections to martingales from \cite{BSV} and \cite{O5}, Section 4.2. 

\subsection{Laminates} Let  $\R^{m\times n}$ denote  the space of all real matrices of dimension  $m\times n$ and  $\R^{n\times n}_{sym}$ denote  the subclass of $\R^{n\times n}$ which consists of all real symmetric $n\times n$ matrices.

\begin{dfn}
A function $f:\R^{m\times n} \to \R$ is said to be {\it rank-one convex}, if  for all $A,B \in \R^{m\times n}$ with $\textrm{rank }B= 1$, the function 
$t\mapsto f(A+tB)$ is convex
\end{dfn}

For other equivalent definitions of rank-one convexity, see \cite[p.~100]{Dac},  
Now let  $\mathcal{P}=\mathcal{P}(\R^{m\times n})$ be the class of all compactly supported probability measures on $\R^{m \times n}$.
For $\nu \in \mathcal{P}$, we denote by $$\overline{\nu} = \int_{\R^{m\times n}}{X d\nu(X)}$$ the \emph{center of mass} or \textit{barycenter} of $\nu.$

\begin{dfn}
We say that a measure $\nu \in \mathcal{P}$ is a \textit{laminate} (and write $\nu\in\mathcal{L}$), if 
\begin{equation*}
f(\overline{\nu}) \leq \int_{\R^{m\times n}}f \mbox{d}\nu
\end{equation*} 
for all rank-one convex functions $f$. The set of laminates with barycenter $0$ is denoted by $\mathcal{L}_0(\R^{m\times n})$. 
\end{dfn}
Laminates can be used to obtain lower bounds for solutions of certain PDEs, as was first noticed by  Faraco in \cite{F}. Furthermore, laminates arise naturally in several applications of convex integration, where can be used to produce interesting counterexamples, see e.g. \cite{AFS}, \cite{CFM}, \cite{KMS}, \cite{MS99} and \cite{SzCI}. For our results here we will be interested in the case of $2\times 2$ symmetric matrices. An important observation to make is  that laminates can be regarded as probability measures that record the distribution of the gradients of smooth maps as described by Corollary \ref{coro} below. We briefly explain this and refer the reader \cite{Kirchheim}, \cite{MS99} and \cite{SzCI} for full details.  

\begin{dfn}
Let $U$ be a subset of $\R^{2\times 2}$ and let $\mathcal{PL}(U)$ denote  the smallest
class of probability measures on $U$ which 

\begin{itemize}

\item[(i)] contains all measures of the form $\lambda \delta_A+(1-\lambda)\delta_B$ with $\lambda\in [0,1]$ and satisfying $\textrm{rank}(A-B)=1$;

\item[(ii)] is closed under splitting in the following sense: if $\lambda\delta_A+(1-\lambda)\nu$ belongs to $\mathcal{PL}(U)$ for some $\nu\in\mathcal{P}(\R^{2\times 2})$ and $\mu$ also belongs to $\mathcal{PL}(U)$ with $\overline{\mu}=A$, then also $\lambda\mu+(1-\lambda)\nu$ belongs to $\mathcal{PL}(U)$.
\end{itemize}

The class $\mathcal{PL}(U)$ is called the  \emph{prelaminates} in $U$. 
\end{dfn} 

It follows immediately from the definition that the class $\mathcal{PL}(U)$ only contains atomic measures. Also, by a successive application of Jensen's inequality, we have the inclusion $\mathcal{PL}\subset\mathcal{L}$. The following are two well known lemmas in the theory of laminates; see \cite{AFS}, \cite{Kirchheim}, \cite{MS99}, \cite{SzCI}. 

\begin{lemma}
Let $\nu=\sum_{i=1}^N\lambda_i\delta_{A_i}\in\mathcal{PL}(\R^{2\times 2}_{sym})$ with $\overline{\nu}=0$. Moreover, let
$0<r<\tfrac{1}{2}\min|A_i-A_j|$ and $\delta>0$. For any bounded domain $\calB\subset\R^2$ there exists $u\in W^{2,\infty}_0(\calB)$ such that $\|u\|_{C^1}<\delta$ and for all $i=1\dots N$
$$
\bigl|\{x\in\calB:\,|D^2u(x)-A_i|<r\}\bigr|=\lambda_i|\calB|.
$$
\end{lemma}

\begin{lemma}
Let $K\subset\R^{2\times 2}_{sym}$ be a compact convex set and $\nu\in\mathcal{L}(\R^{2\times 2}_{sym})$ with $\operatorname*{supp}\nu\subset K$. For any relatively open set $U\subset\R^{2\times 2}_{sym}$ 
with $K\subset\subset U$,  there exists a sequence $\nu_j\in \mathcal{PL}(U)$ of prelaminates with $\overline{\nu}_j=\overline{\nu}$ and $\nu_j\overset{*}{\rightharpoonup}\nu$, where $\overset{*}{\rightharpoonup}$ denotes weak convergence of measures. 
\end{lemma}

Combining these two lemmas and using a simple mollification, we obtain the following statement, proved by Boros, Sh\'ekelyhidi Jr. and Volberg \cite{BSV}. It exhibits the connection  between laminates supported on symmetric matrices and second derivatives of functions.  This Corollary will play a crucial role in our argumentation below. Recall that (as in the introduction) we use  $\bD$ to denote the unit disc of $\mathbb{C}$.

\begin{corr}\label{coro}
Let $\nu\in\mathcal{L}_0(\R^{2\times 2}_{sym})$. Then there exists a sequence $u_j\in C_0^{\infty}(\bD)$ with uniformly bounded second derivatives, such that
\begin{equation}\label{gby}
\frac{1}{|\bD|}\int_{\bD} \phi(D^2u_j(x))\,\mbox{d}x\,\to\,\int_{\R^{2\times 2}_{sym}}\phi\,\mbox{d}\nu
\end{equation}
for all continuous $\phi:\R^{2\times 2}_{sym}\to\R$. 
\end{corr}

The above corollary works for laminates of barycenter $0$. This gives rise to certain technical difficulties, as the natural laminates induced by the martingale examples of Section 2 do not have this property. To overcome the problem, we will have to center the examples as done below.

\subsection{Biconvex functions and a special laminate} 
The next step in our analysis is devoted to the introduction of a certain special laminate. We need some additional notation. 
A function $\zeta:\R\times \R\to \R$ is said to be \emph{biconvex} if for any fixed $z\in \R$, the functions $x\mapsto \zeta(x,z)$ and $y\mapsto \zeta(z,y)$ are convex. Now, for $\lambda\leq 1/2$, let $f$, $g$ be the martingales of Section 2, which exhibit the sharpness of \eqref{martin} (if $\lambda$ is strictly smaller then $1/2$, then there is a whole family of examples, corresponding to different choices of $\kappa$ and $N$ - these two parameters will be specified later). Consider the $\R^2$-valued martingale
$$ (F,G):=\left(\frac{f+g-\kappa}{2},\frac{f-g-\kappa}{2}\right).$$
We subtract $\kappa$ on both coordinates to ensure that the pair $(F,G)$ has mean $(0,0)$. 
This sequence has the following \emph{zigzag} property: for any $n\geq 0$ we have $F_n=F_{n+1}$ with probability $1$ or $G_n=G_{n+1}$ almost surely; that is, in each step $(F,G)$ moves either vertically, or horizontally. Indeed, this follows directly from the construction that for $n\geq 1$ we have $\mathbb{P}(df_n= dg_n)=1$ or $\mathbb{P}(df_n=-dg_n)=1$. This property combines nicely with biconvex functions: if $\zeta$ is such a function, then a successive application of Jensen's inequality gives
\begin{equation}\label{dm} 
\E \zeta(F_{n},G_{n})\geq \E\zeta(F_{n-1},G_{n-1})\geq \ldots\geq \E \zeta(F_0,G_0)=\zeta(0,0).
\end{equation}
The martingale $(F,G)$, or rather the distribution of its terminal variable $(F_{\infty},G_{\infty})$, gives rise to a probability measure $\nu$ on $\R^{2\times 2}_{sym}$: put
$$ \nu\left(\diag(x,y)\right)=\mathbb{P}\big((F_{\infty},G_{\infty})=(x,y)\big),\qquad (x,y)\in \R^2.$$
Here and below, $\diag(x,y)$ denotes the diagonal matrix 
$ \left(\begin{array}{cc}
x & 0\\ 
0 & y
\end{array}\right).$ 
The key observation is that $\nu$ is a laminate of barycenter $0$. To prove this, note that if $\psi:\R^{2\times 2}$ is a rank-one convex, then $(x,y)\mapsto \psi(\operatorname*{diag}(x,y))$ is biconvex and thus, by \eqref{dm}, 
\begin{align*}
 \int_{\R^{2\times 2}}\psi\mbox{d}\nu
&=\E \psi(\operatorname*{diag}(F_{\infty},G_{\infty})) \geq \psi(\operatorname*{diag}(0,0))=\psi(\bar{\nu}).
\end{align*}
Finally, note that $\mathbb{P}\big(F_{\infty}+G_{\infty}\in \{-\kappa,1-\kappa\}\big)=\mathbb{P}(f_{\infty}\in\{0,1\})=1$, and hence the support of $\nu$ is contained in 
\begin{equation}\label{defK}
 K=\big\{\diag(x,y):x+y\in \{-\kappa,1-\kappa\}\big\}.
\end{equation}

\subsection{Proof of \eqref{sharp}}
For the convenience of the reader,  we give an outline here of the ideas behind the arguments below. 
We start with the application Corollary \ref{coro} to the laminate $\nu$: let $(u_j)_{j\geq 1}$ be the corresponding sequence of smooth functions. As we have just observed above, the support of $\nu$ is contained in $K$ given by \eqref{defK}. Since the distribution of $u_j$ is close to $\nu$ (in the sense of Corollary \ref{coro}), we expect that $\Delta u_j$, essentially, takes only values close to $-\kappa$ or close to $1-\kappa$. Thus, if we define  $v_j=\Delta u_j+\kappa \chi_\bD$ for $j=1,\,2,\,\ldots$, then $v_j$ is close to an indicator function of a certain set $E$. Thus, to prove the sharpness of \eqref{sharp}, one can try to study this estimate with $\chi_E$ replaced by $v_j$. We will look separately at the action on $\Re{B}$ on $\Delta u_j$ and $\kappa\chi_{\bD}$.  To handle the Laplacian, we will use the arguments from the previous two subsections, and the term $\kappa \chi_\bD$ will be dealt with directly.

\smallskip

\emph{Step 1.} Fix $\lambda\in (0,1/2)$ and pick the positive number $M=\eta C(\lambda)$ so that  $M<C(\lambda)$.  Then, by the reasoning presented in Section 2, if $\kappa>0$ is sufficiently small and an integer $N$ is large enough, then the corresponding martingale  $(f,g)$ constructed there satisfies $\E(|g_\infty|-\lambda)_+(1-f_\infty)>M \E f_0=M\kappa$. 
Next, let $\e\in (0,1/4)$ be a given number (which will be eventually sent to $0$). In what follows, $C_1$, $C_2$, $C_3$, $\ldots$ will denote constants which depend only on $\kappa$ and $N$.

\smallskip

\emph{Step 2.} Consider a continuous function $\phi:\R^{2\times 2}_{sym}\to \R$ given by $\phi(\diag(x,y))=|x+y+\kappa|$. By Corollary \ref{coro}, and using the fact that $F_\infty+G_\infty+\kappa=f_\infty$, we have
$$ \frac{1}{|\bD|}\int_{\bD}|v_j|=\frac{1}{|\bD|}\int_{\bD}\phi(D^2u_j)\xrightarrow{j\to\infty}\int_{\R^{2\times 2}_{sym}}\phi\mbox{d}\nu
=\E|F_{\infty}+G_{\infty}+\kappa|=\kappa.$$ 
Thus  for sufficiently large $j$,
 \begin{equation}\label{ss1}
 \frac{1}{|\bD|}\int_{\bD}|v_j|\leq \kappa(1+\e).
 \end{equation}

\emph{Step 3.} Consider a continuous function $\phi:\R^{2\times 2}_{sym}\to [0,1]$, satisfying $\phi(\diag(x,y))=1$ when $|x+y-1+\kappa|>\e$ and $|x+y+\kappa|>\e$, and $\phi(\diag(x,y))=0$,  if $x+y+\kappa\in \{0,1\}$. By Corollary \ref{coro},
\begin{equation}\label{diss}
 \frac{1}{|\bD|}\int_\bD\phi(D^2u_j) \to \int_{\R^{2\times 2}_{sym}}\phi\mbox{d}\nu=0,
\end{equation}
since $\mathbb{P}(F_{\infty}+G_{\infty}+\kappa\in \{0,1\})=1$. Consider the sets
$$ {E^j}=\{x\in\bD:|\Delta u_j(x)-1+\kappa|\leq \e\} \quad \text{and} \quad \tilde{{E^j}}=\{x\in\bD:|\Delta u_j(x)+\kappa|\leq \e\}.$$
Then \eqref{diss} implies that
\begin{equation}\label{diss2}
\frac{|\bD\setminus({E^j}\cup\tilde{{E^j}})|}{|\bD|}<\e, \qquad \mbox{for sufficiently large $j$}.
\end{equation}

\emph{Step 4.} Next, consider a continuous function $\phi:\R^{2\times 2}_{sym}\to [0,1]$ which satisfies $\phi(\diag(x,y))=1$ if $x+y+\kappa=1$ and $\phi(\diag(x,y))=0$,  if $|x+y+\kappa-1|>\e$. Then
\begin{equation}\label{beth}
 |{E^j}|\geq \int_\bD\phi(D^2u_j)\xrightarrow{j\to\infty} |\bD|\int_{\R^{2\times 2}_{sym}}\phi\mbox{d}\nu=|\bD|\mathbb{P}(F_{\infty}+G_{\infty}+\kappa=1)=|\bD|\kappa.
 \end{equation}
 An analogous argument, exploiting the function $\phi$ which is $1$ when $|x+y-1+\kappa|\leq \e$ and vanishes for $|x+y-1+\kappa|\geq 2\e$, yields
 $ \limsup_{j\to \infty}|{E^j}|\leq |\bD|\kappa$. As we have mentioned in Step 1, the numbers $\kappa$ we consider are small; this implies that the sets $E^j$ we obtain are of small measure (and in particular, satisfy $|E^j|<|\bD|/2$). 
 
By \eqref{beth}, we get that for any $1\leq q<\infty$ and large $j$,
\begin{align*}
 &||v_j-\chi_{E^j}||_{L^q(\R^2)}^q\\
&=||\Delta u_j+\kappa-\chi_{E^j}||_{L^q(\bD)}^q\\
 &=\int_{E^j} |\Delta u_j+\kappa-\chi_{E^j}|^q+\int_{\tilde{{E^j}}}|\Delta u_j+\kappa-\chi_{E^j}|^q+\int_{\bD\setminus({E^j}\cup\tilde{{E^j}})}|\Delta u_j+\kappa-\chi_{E^j}|^q\\
&\leq \e^q|{E^j}|+\e^q|\tilde{{E^j}}|+\e|\bD|(\sup_{\bD}|\Delta u_j|+\kappa).
\end{align*}
Here in the last passage we have used the definition of ${E^j}$, $\tilde{{E^j}}$ and \eqref{diss2}. Combining this with \eqref{beth} (and the fact that the second-order partial derivatives of $u_j$ are uniformly bounded, see Corollary \ref{coro}), we get that for sufficiently large $j$,
\begin{equation}\label{fory}
 ||v_j-\chi_{E^j}||_{L^q(\R^2)}^q\leq C_1\e|{E^j}|.
 \end{equation}
Thus, the function $v_j$ is close to the indicator function of ${E^j}$.

\smallskip

\emph{Step 5.} Next, consider the function $\phi:\R^{2\times 2}_{sym}\to \R$ given by $\phi(\diag(x,y))=(|x-y|-\lambda)_+(1-x-y-\kappa)$. By the choice of $\kappa$, $N$ and \eqref{ss1},
\begin{align*}
 \frac{1}{|\bD|}\int_{\bD}\phi(D^2u_j)\xrightarrow{j\to\infty}\int_{\R^{2\times 2}_{sym}}\phi\mbox{d}\nu&=\E(|g_{\infty}|-\lambda)_+(1-f_\infty)\\
&>M\kappa\\
&\geq \frac{M}{1+\e}\cdot \frac{1}{|\bD|}\int_{\bD}|v_j|\\
&\geq \frac{M}{1+\e}\cdot\frac{1}{|\bD|}\left(|{E^j}|-\int_{\bD}|v_j-\chi_{E^j}|\right).
\end{align*}
Now multiply throughout by $|\bD|$ and apply \eqref{fory} with $q=1$ to get  that for sufficiently large $j$,
$$  \int_{\bD}\phi(D^2u_j)\geq \frac{M}{1+\e}(1-C_1\e)|{E^j}|.$$
However, observe that
\begin{align*}
 \phi(D^2u_j)&=(|\partial_{11}u_j-\partial_{22}u_j|-\lambda)_+(1-\Delta u_j-\kappa)\\
&=(|\Re{B}\Delta u_j|-\lambda)_+(1-v_j)\\
&=(|\Re{B}v_j|-\lambda)_+(1-v_j)
\end{align*}
on $\bD$. In the last line we have used the fact that $B\chi_\bD=0$ on $\bD$.  Actually, as already mentioned in \eqref{B-disc},   
$ {B}\chi_\bD(z)=-\chi_{\mathbb{C}\setminus \bD}(z)/z^2$. 
Hence, the preceding considerations yield that for large $j$,
\begin{align*}
\frac{M}{1+\e}(1-C_1\e)|{E^j}|&\leq \int_\bD(|\Re{B}v_j|-\lambda)_+(1-v_j)\leq I_1+I_2+I_3,
\end{align*}
where
\begin{align*}
I_1&=\int_\bD(|\Re{B}v_j|-\lambda)_+(\chi_{E^j}-v_j),\\
I_2&=\int_\bD|\Re{B}(\chi_{E^j}-v_j)|(1-\chi_{E^j}),\\
I_3&=\int_\bD(|\Re{B}(\chi_{E^j})|-\lambda)_+(1-\chi_{E^j}).
\end{align*}
Since ${B}$ is an isometry on $L^2(\R^2)$, an application of Schwarz inequality, \eqref{fory} and  \eqref{beth} give that  $I_1+I_2\leq C_2\e^{1/2}|{E^j}|$.  Putting all the above facts together, we get that if $j$ is sufficiently large, then
$$ \int_\bD(|\Re{B}(\chi_{E^j})|-\lambda)_+(1-\chi_{E^j})\geq \frac{M}{1+\e}(1-C_3\e^{1/2})|{E^j}|.$$
Therefore, if we consider the set $D^j=\{z \in \bD: |\Re{B}(\chi_{E^j})|\geq \lambda \}\cup E^j$, we obtain
\begin{align*}
 \int_{D^j\setminus {E^j}} |\Re{B}(\chi_{E^j})|&\geq \frac{M}{1+\e}(1-C_3\e^{1/2})|{E^j}|+\lambda |D^j\setminus {E^j}|\\
 &\geq \frac{(1-C_3\e^{1/2})\eta}{1+\e}\big[C(\lambda)|{E^j}|+\lambda|D^j\setminus {E^j}|\big].
\end{align*}
A simple analysis of the derivative in $\lambda$ shows that the right-hand side is not smaller than
$$ \frac{(1-C_3\e^{1/2})\eta}{1+\e}\left[|{E^j}|+|{E^j}|\ln\left(\frac{|D^j|}{2|{E^j}|}\right)\right]=\frac{(1-C_3\e^{1/2})\eta}{1+\e}\left[|{E^j}|\ln\left(\frac{e|D^j|}{2|{E^j}|}\right)\right].$$
Using the fact that  $\e>0$ was arbitrary, this completes the proof of Theorem \ref{sharpness} and shows the  optimality of \eqref{mainin}. \qed

\bigskip

\subsection*{Acknowledgments}  Rodrigo Ba\~nuelos gratefully acknowledges useful conversations with Eero Saksman on the subject of this paper.


\begin{thebibliography}{99}

\bibitem{AppBan} D. Applebaum and R. Ba\~nuelos,  {\it Martingale transform and L\'evy Processes on Lie Groups,} {\bf submitted}. 


\bibitem{Ast}{K. Astala, {\it Area distortion of quasiconformal mappings}, Acta. Math. {\bf 173} (1994), 37--60.}


\bibitem{AstIwaMar} K. Astala, T. Iwaniec and G. Martin, {\it Elliptic Partial Differential Equations and Quasiconformal Mappings in the Plane,} Princeton University Press, 2009. 
\bibitem{AFS} { K.~Astala, D.~Faraco,  L.~Sz{\'e}kelyhidi, Jr.}, {\em Convex integration and the $L^p$ theory of elliptic equations}. Ann. Sc. Norm. Super. Pisa Cl. Sci. (5), {\bf 7} (2008), pp. 1--50.

\bibitem{Ban1} R. Ba\~nuelos, {\it The foundational inequalities of D. L. Burkholder and some of their ramifications,} { Illinois J. Math.} {\bf 54} (2010), pp. 789--868.

\bibitem{BB}{R. Ba\~nuelos and K. Bogdan, {\it L\'evy processes and Fourier multipliers}, J. Funct. Anal. \textbf{250} (2007),
pp. 197--213.}

\bibitem{BBB}{R. Ba\~nuelos, A. Bielaszewski and K. Bogdan, {\it Fourier multipliers for non-symmetric L\'evy
processes}, Marcinkiewicz centenary volume,
Banach Center Publ. \textbf{95} (2011), pp. 9--25, Polish Acad. Sci. Inst. Math., Warsaw.}

\bibitem{BJ}{R. Ba\~nuelos and P. Janakiraman, {\it $L_p$-bounds for the Beurling-Ahlfors transform}, Trans. Amer. 
Math. Soc. \textbf{360} (2008), pp. 3603--3612.}

\bibitem{BMH}{R. Ba\~nuelos and P. M\'endez-Hern\'andez, {\it Space-time Brownian motion and the Beurling-Ahlfors
transform}, Indiana University Math. J. \textbf{52} (2003), pp. 981--990.}

\bibitem{BO1} R. Ba\~nuelos and A. Os\c ekowski,  {\it Burkholder inequalities for submartingales, Bessel processes and conformal martingales}, American Journal of Mathematics, {\bf to appear}.

\bibitem{BO}{R. Ba\~nuelos and A. Os\c ekowski, {\it Martingales and sharp bounds for Fourier multipliers}, Ann. Acad. Sci.
Fenn. Math. \textbf{37} (2012), pp. 251--263.}

\bibitem{BW}{R. Ba\~nuelos and G. Wang, {\it Sharp inequalities for martingales with applications to the Beurling-Ahlfors and Riesz transformations}, Duke Math. J. \textbf{80} (1995), pp. 575--600.}

\bibitem{Bo}{B. Bojarski, {\it Homeomorphic solutions of Beltrami systems}, Dokl. Acad. Nauk SSSR \textbf{102} (1955), pp. 661--664.}

\bibitem{Bo2}{B. Bojarski, {\it Generalized solutions of a system of differential equations of elliptic type with discontinuous coefficients}, Math. Sb. \textbf{43 (85)} (1957), pp. 451--503.}

\bibitem{BJV1} A. Borichev, P. Janakiraman and A. Volberg, {\it On Burkholder function for orthogonal martingales and zeros of Legendre polynomials}, Amer. Jour. Math. {\bf 135} (2013), pp. 207--236. 

\bibitem{BorJanVol} A.  Borichev, P. Janakiraman, A. Volberg, {\it Subordination by orthogonal martingales in $L^{p}$ and zeros of Laguerre polynomials}, Duke Math. J. \textbf{162} (2013), pp. 889--924. 

\bibitem{BSV}{N. Boros, L. Sz\'ekelyhidi Jr. and A. Volberg, {\it Laminates meet Burkholder functions}, Journal de Math\'ematiques Pures et Appliqu\'ees, {\bf to appear}.}

\bibitem{B0}
{D. L. Burkholder, 
{\it  Boundary value problems and sharp inequalities for martingale transforms}, 
Ann. Probab. 12 (1984), pp. 647--702.}

\bibitem{B1}{D. L. Burkholder, {\it Explorations in martingale theory and its applications}, \'Ecole d'Ete de Probabilit\'es de Saint-Flour XIX---1989, pp. 1--66, Lecture Notes in Math., 1464, Springer, Berlin, 1991.}

\bibitem{CFM}{S.~Conti, D.~Faraco, F.~Maggi}, {\em A new approach to counterexamples to {$L^1$} estimates: {K}orn's inequality, geometric rigidity, and regularity for gradients of separately convex functions}, Arch. Rat. Mech. Anal. {\bf 175(2)} (2005), pp. 287--300.

\bibitem{Dac} B. Dacoronga, {\it Direct Methods in the Calculus of Variations,} Springer
1989.

\bibitem{DM}{C. Dellacherie and P.-A. Meyer, {\it Probabilities and 
potential B: Theory of martingales}, North Holland, Amsterdam, 1982.}

\bibitem{EH}{A. Eremenko and D. Hamilton, {\it On the area distortion by quasiconformal mappings}, Proc. Amer. Math. Soc.
\textbf{123} (1995), pp. 2793--2797.}

\bibitem{F}{D. Faraco, {\em Milton's conjecture on the regularity of solutions to isotropic equations},
Ann. Inst. H. Poincar\'e Anal. Non Lin\'eaire, \textbf{20} (2003), pp. 889--909.}

\bibitem{GR}{F. W. Gehring and E. Reich, {\it Area distortion under quasiconformal mappings}, Ann. Acad. Sci. Fenn. Ser. A.I 
\textbf{388} (1966), pp. 1--14.}

\bibitem{GMS}{E. Geiss, S. Mongomery-Smith, E. Saksman, {\it On singular integral and martingale transforms},
Trans. Amer. Math. Soc. \textbf{362} (2010), pp. 553-575.}

\bibitem{Iw} T. Iwaniec, \emph{Extremal inequalities in Sobolev spaces and quasiconformal mappings,} Z. Anal. 
Anwendungen {\bf 1} (1982), pp. 1--16.

\bibitem{Jan} P. Janakiraman, {\it Orthogonality in complex martingale spaces and connections to the Beurling-Ahlfors operator,} 
 Illinois J. Math. {\bf 54} no. 4  (2010), pp. 1509-1563.


\bibitem{Kirchheim}{B.~Kirchheim}, {\em Rigidity and Geometry of Microstructures}, Habilitation Thesis, University of Leipzig (2003), {http://www.mis.mpg.de/publications/other-series/ln/lecturenote-1603.html}

\bibitem{KMS}{ B.~Kirchheim, S.~M{\"u}ller, V.~{\v{S}}ver{\'a}k}, {\em Studying nonlinear pde by geometry in matrix space}, Geometric Analysis and nonlinear partial differential equations, Springer (2003), pp. 347--395.

\bibitem{MS99}{ S.~M\"uller, V.~\v Sver\'ak}, {\em Convex integration for {L}ipschitz mappings and counterexamples to regularity}, Ann. of Math. (2), {\bf 157} no. 3 (2003), pp. 715--742.  

\bibitem{NV} F. Nazarov and  A. Volberg, {\it  Heat extension of the Beurling operator and estimates for its norm,} St. Petersburg Math. J. {\bf 15}, (2004), pp. 563--573.

\bibitem{O0}{A. Os\c ekowski, {\it  Sharp martingale and semimartingale inequalities}, Monografie Matematyczne \textbf{72} (2012), Birkh\"auser Basel.}

\bibitem{O3}{A. Os\c ekowski, {\it Logarithmic inequalities for Fourier multipliers}, Math. Z. \textbf{274} (2013), pp. 515--530.} 

\bibitem{O4}{A. Os\c ekowski, {\it Weak-type inequalities for Fourier multipliers with applications to the Beurling-Ahlfors transform}, J. Math. Soc. Japan, {\bf to appear}.} 

\bibitem{O5} A. Os\c ekowski, {\it On restricted weak-type constants of Fourier multipliers}, 
Publicacions Matem\`atiques,  {\bf to appear}. 

\bibitem{SzCI}{L.~Sz{\'e}kelyhidi, Jr.}, {\em Counterexamples to elliptic regularity and convex integration}, Contemp. Math. {\bf 424} (2007), pp. 227--245.

\bibitem{W}{G. Wang, {\it Differential subordination and strong differential subordination for
continuous time martingales and related sharp inequalities}, Ann. Probab. \textbf{23} 
(1995), pp. 522--551.}
\end{thebibliography}
\end{document}